\documentclass[12pt,leqno]{amsart}\usepackage[latin9]{inputenc}\usepackage{verbatim,amsthm,amstext,amssymb,color,amscd,bm}\makeatletter
\usepackage[colorlinks=true]{hyperref}\usepackage[shortlabels]{enumitem}\hypersetup{urlcolor=red, citecolor=blue}
\numberwithin{equation}{section}
\theoremstyle{plain}\newtheorem{theorem}{Theorem}[section]
\newtheorem{lemma}[theorem]{Lemma}\newtheorem{proposition}[theorem]{Proposition}
\theoremstyle{definition}
\theoremstyle{definition}\newtheorem{remark}[theorem]{Remark}
\theoremstyle{definition}
\theoremstyle{definition}
\theoremstyle{definition}\newtheorem{question}[theorem]{Question}
\DeclareMathOperator{\Widim}{Widim}

\DeclareMathOperator{\mdim}{mdim}
\makeatother
\begin{document}

\title[A non-embeddable example for amenable group actions]
{Mean dimension and a non-embeddable example\\for amenable group actions}

\author{Lei Jin}
\address{Lei Jin:
Center for Mathematical Modeling, University of Chile and UMI 2807 - CNRS}
\email{jinleim@impan.pl}

\author{Kyewon Koh Park}
\address{Kyewon Koh Park:
Korea Institute for Advanced Study,
85 Hoegiro, Dongdaemun-gu, Seoul 02455, Korea}
\email{kkpark@kias.re.kr}

\author{Yixiao Qiao}
\address{Yixiao Qiao (corresponding author):
School of Mathematical Sciences,
South China Normal University,
Guangzhou 510631, China}
\email{yxqiao@mail.ustc.edu.cn}


\begin{abstract}
For every infinite (countable discrete) amenable group $G$ and every positive integer $d$ we construct a minimal $G$-action of mean dimension $d/2$ which cannot be embedded in the full $G$-shift on $([0,1]^d)^G$.
\end{abstract}

\maketitle

\bigskip

\section{Introduction}
Mean dimension, which was introduced by Gromov \cite{Gromov} in 1999, is a numerical topological invariant of dynamical systems. As an analogue of topological dimension, its advantage has now been shown in the study of dynamical systems whose topological entropy equals infinity and whose phase space has infinite topological dimension. In particular, mean dimension is intimately involved in the \textit{embedding problem}.

We say that a dynamical system can be \textbf{embedded} in another if the former is topologically conjugate to a subsystem of the latter (see Subsection \ref{subsec:groupaction} for a formal definition). Let $G$ be an infinite countable discrete amenable group and $d$ a positive integer. The \textbf{full $G$-shift $\sigma$ on $([0,1]^{d})^{G}$} (or the \textit{shift} on $([0,1]^{d})^{G}$ for short) is defined by
$$\sigma:G\times([0,1]^{d})^{G}\to([0,1]^{d})^{G},\quad(g,(x_{h})_{h\in G})\mapsto(x_{hg})_{h\in G}.$$
In general, the embedding problem is to decide if a $G$-action can be embedded in the full $G$-shift on $([0,1]^{d})^{G}$.

For a (possibly) simpler picture, we may consider $G=\mathbb{Z}$ tentatively. An easy observation is that if a $\mathbb{Z}$-action can be embedded in the shift on $[0,1]^\mathbb{Z}$ then it cannot possess too many (in the sense of topological dimension) periodic points. The first successful attempt on the embedding problem for $\mathbb{Z}$-actions was made by Jaworski \cite{Jaworski} in 1974, whose result states that if a dynamical system $(X,\mathbb{Z})$ has no periodic points and if the space $X$ is finite dimensional, then $(X,\mathbb{Z})$ can be embedded in the shift on $[0,1]^\mathbb{Z}$. However, for the case that $X$ is an infinite dimensional space, the situation becomes much more complicated. In particular, embeddability of minimal dynamical systems in the shift on $[0,1]^\mathbb{Z}$ attracts extensive attention\footnote{Notice that a minimal $\mathbb{Z}$-action must have no periodic points unless its phase space is a finite set.}. In 2000, Lindenstrauss and Weiss \cite{Lindenstrauss--Weiss} developed mean dimension theory in dynamical systems, and especially, in connection with the embedding problem. They asserted that if a $\mathbb{Z}$-action can be embedded in the shift on $[0,1]^\mathbb{Z}$ then its mean dimension must be at most $1$; and meanwhile, they constructed a minimal $\mathbb{Z}$-action of mean dimension strictly greater than $1$. It follows immediately that not every minimal $\mathbb{Z}$-action can be embedded in the shift on $[0,1]^\mathbb{Z}$.

In the theory of topological dimension, the celebrated Menger--N\"obeling theorem asserts that any compact metric space of topological dimension strictly less than $n/2$ can be topologically embedded into $[0,1]^{n}$, where $n\in\mathbb{N}$ (see \cite{Hurewicz--Wallman} for details). This theorem is sharp, and naturally motivates a converse question for minimal dynamical systems. Let us state it precisely in the context of general group actions:

\begin{question}\label{minimalembedding}
Let $G$ be an infinite countable discrete amenable group and $d$ a positive integer. Determine the optimal value of constants $C\in[0,+\infty]$ such that the following assertion is true: If a minimal $G$-action has mean dimension strictly less than $C$, then it can be embedded in the full $G$-shift on $([0,1]^{d})^{G}$.
\end{question}

\begin{remark}
As an analogue of the Menger--N\"obeling embedding theorem, the assumption of amenability of $G$ in Question \ref{minimalembedding} is to guarantee that mean dimension of any $G$-action is situated in $[0,+\infty]$ (see \cite{Li}). The optimal value of such constants $C$ exists in $[0,+\infty]$ as well, because $C=0$ is in fact a trivial constant that makes the assertion true.
\end{remark}

An amazing result in this direction was due to Lindenstrauss \cite{Lindenstrauss} in 1999, who showed that if a minimal $\mathbb{Z}$-action has mean dimension strictly less than $d/36$ then it can be embedded in the shift on $([0,1]^{d})^\mathbb{Z}$. In 2014, Lindenstrauss and Tsukamoto \cite{Lindenstrauss--Tsukamoto} constructed a nice example of a minimal $\mathbb{Z}$-action of mean dimension equal to $d/2$, which cannot be embedded in the shift on $([0,1]^{d})^{\mathbb{Z}}$. This construction indicates that the answer to Question \ref{minimalembedding} is not larger than
$d/2$ in the setting of $G=\mathbb{Z}$. In 2015, going through harmonic and complex analysis, Gutman and Tsukamoto \cite{Gutman--Tsukamoto} proved a significant result: If a minimal $\mathbb{Z}$-action has mean dimension strictly less than $d/2$, then it can be embedded in the shift on $([0,1]^{d})^{\mathbb{Z}}$. Thus, the solution to Question \ref{minimalembedding} for $\mathbb{Z}$-actions is $d/2$.

However, if we proceed to a further stage $G=\mathbb{Z}^{k}$ ($k\in\mathbb{N}$), then we encounter serious difficulties. We refer to \cite{Gutman--Lindenstrauss--Tsukamoto} and \cite{Gutman--Qiao--Tsukamoto} for detailed explanations, ideas and techniques. Nevertheless, it turns out \cite{Gutman--Qiao--Tsukamoto} that $d/2$, as anticipated, is still the exact solution to Question \ref{minimalembedding} for the case $G=\mathbb{Z}^{k}$ (where $k\in\mathbb{N}$).\footnote{For related results see \cite{Gutman,Gutman--Qiao--Szabo,Gutman--Tsukamoto1}.}

In contrast to $\mathbb{Z}^{k}$-actions, there has been no essential progress with Question \ref{minimalembedding} in general settings. Crucial problems will definitely arise due to geometric structures of general groups different from $\mathbb{Z}^{k}$. However, it is reasonable to expect $d/2$ to be the solution to Question \ref{minimalembedding} for amenable group actions. The main result of the present paper is to confirm this assertion from above: The solution to Question \ref{minimalembedding} does not exceed $d/2$.\footnote{We would like to remind the reader that it is still unknown yet whether $d/2$ is the optimal.}

\begin{theorem}\label{maintheorem}
Let $G$ be an infinite countable discrete amenable group and $d$ a positive integer. Then there is a minimal $G$-action $(X,G)$ whose mean dimension is equal to $d/2$ such that $(X,G)$ cannot be embedded in the full $G$-shift on $([0,1]^{d})^{G}$.
\end{theorem}

\bigskip

This paper is organized as follows. In Section \ref{sec:preliminaries}, we gather basic notions in amenable group actions and mean dimension; to prepare our proof we also collect fundamental tools and necessary propositions, especially including tilings of amenable groups. In Section \ref{sec:maintheorem}, we provide a constructive proof of Theorem \ref{maintheorem}.

\bigskip

\textbf{Acknowledgements.}
A part of this research was done when Yixiao Qiao visited the Korea Institute for Advanced Study (KIAS) in 2018. The authors would like to thank Professor Dou Dou, Professor Tomasz Downarowicz, Professor Yonatan Gutman, Professor Masaki Tsukamoto, and Professor Guohua Zhang for their warm comments and insightful suggestions, as well as helpful discussions. L. Jin was supported by Basal Funding AFB 170001 and Fondecyt Grant No. 3190127, and was partially supported by NNSF of China No. 11971455. Y. Qiao was supported by NNSF of China No. 11901206.

\bigskip

\section{Preliminaries}\label{sec:preliminaries}
\subsection{Group actions}\label{subsec:groupaction}
Throughout this paper, by a \textbf{$G$-action} we always understand a triple $(X,G,\Phi)$, where $X$ is a compact metric space, $G$ is an infinite countable discrete amenable\footnote{The terminology of amenability is planned to be presented in the next subsection.} group with the identity element $e$, and $$\Phi:G\times X\to X,\quad(g,x)\mapsto\Phi(g,x)$$ is a continuous mapping satisfying that
$$\Phi(e,x)=x,\quad\quad\Phi(gh,x)=\Phi(g,\Phi(h,x)),\quad\forall x\in X,\;\forall g,h\in G.$$
Usually, $(X,G,\Phi)$ and $\Phi(g,x)$ are abbreviated to $(X,G)$ and $gx$, respectively.

\medskip

Let $(X,G)$ be a $G$-action. For a subset $F$ of $G$ and a point $x\in X$, we set $$Fx=\{gx:g\in F\}\subset X.$$ We say that $(X,G)$ is \textbf{minimal} if for every $x\in X$, its \textit{orbit} $Gx$ is dense in $X$. A subset $S$ of $G$ is called \textbf{syndetic} if there exists a finite subset $F$ of $G$ such that $G=FS$, where $FS=\{fs:f\in F,s\in S\}$. A point $x\in X$ is said to be \textbf{almost periodic} if for each neighborhood $U$ of $x$, there is a syndetic subset $S$ of $G$ such that $Sx\subset U$. We recall that minimality can be equivalently characterized as follows.

\begin{lemma}[{\cite[Chapter 1]{Auslander}}]\label{minimality}
A $G$-action $(X,G)$ is minimal if and only if $X$ is the orbit closure of an almost periodic point.
\end{lemma}

\medskip

Let $K$ be a compact metric space and $d$ a metric on $K$. We equip $K^G$ with the product topology. A compatible metric $\rho$ on $K^G$ is defined by
\begin{equation}\label{metricrho}
\rho(x,y)=\sum_{g\in G}\alpha_{g}d(x_{g},y_{g}),\quad\forall x=(x_{g})_{g\in G},y=(y_{g})_{g\in G}\in K^{G},
\end{equation}
where $(\alpha_g)_{g\in G}\subset(0,+\infty)$ satisfies
$$\alpha_{e}=1,\quad\sum_{g\in G}\alpha_{g}<+\infty.\footnote{Note that $G$ is countable.}$$

\medskip

The \textbf{full $G$-shift $\sigma$ on $K^{G}$} is the $G$-action $(K^G,\sigma)$ defined by
$$\sigma:G\times K^{G}\to K^{G},\quad(g,(x_{h})_{h\in G})\mapsto(x_{hg})_{h\in G}.\footnote{Notice that the notation $\sigma$ may be kept in different full shifts if there is no ambiguity.}$$
A \textbf{subshift} of $(K^{G},\sigma)$ means a subsystem of the full $G$-shift on $K^{G}$.

\medskip

For $x=(x_{g})_{g\in G}\in K^{G}$ and $F\subset G$ we denote by $$x|_{F}=(x_{g})_{g\in F}\in K^F$$ the restriction of $x$ on $F$, and $$\pi_{F}:K^{G}\to K^{F},\quad x\mapsto x|_{F}$$ the canonical projection mapping. For $p\in K$ we set
$$x(F,p)=\left\{g\in F:x_{g}=p\right\}\subset G.$$

\medskip

Let $(X,G)$ and $(Y,G)$ be two $G$-actions. We say that $(X,G)$ can be \textbf{embedded} in $(Y,G)$ if there is a continuous injective mapping\footnote{Note that this mapping is indeed a homeomorphism of $X$ into $Y$ in our setting.} $f:X\to Y$ such that $f(gx)=gf(x)$ for all $g\in G$ and all $x\in X$. Such a mapping $f$ is called an \textbf{embedding} of $(X,G)$ into $(Y,G)$.

\medskip

\subsection{Tilings of amenable groups}
For a group $G$ we denote by $\mathcal{F}(G)$ the collection of all nonempty finite subsets of $G$. For $T\in\mathcal{F}(G)$ and $\epsilon>0$ we say that a subset $F$ of $G$ is \textbf{$(T,\epsilon)$-invariant} if $$\frac{|B(F,T)|}{|F|}<\epsilon,$$ where $$B(F,T)\footnote{This notation will be used in the sequel.}=\left\{g\in G:Tg\cap F\ne\emptyset,Tg\cap(G\setminus F)\ne\emptyset\right\}$$ and $|\cdot|$ denotes the cardinality of a set.

\medskip

A countable group $G$ is called \textbf{amenable} if there exists a sequence $\{F_{n}\}_{n=1}^{\infty}\subset\mathcal{F}(G)$ such that for any $g\in G$ we have $$\lim_{n\to\infty}\frac{|F_{n}\triangle gF_{n}|}{|F_{n}|}=0.$$ We call such a sequence $\{F_{n}\}_{n=1}^{\infty}$ a \textbf{F\o lner sequence} of the group $G$.

\medskip

An easy observation is that $\{F_{n}\}_{n=1}^{\infty}$ is a F\o lner sequence of $G$ if and only if for any $T\in\mathcal{F}(G)$ and any $\epsilon>0$, $F_{n}$ is $(T,\epsilon)$-invariant for $n$ sufficiently large if and only if for any $T\in\mathcal{F}(G)$ and any $\epsilon>0$, ${|F_{n}\triangle TF_{n}|}/{|F_{n}|}<\epsilon$ for $n$ sufficiently large.

\medskip

Now let $G$ be an infinite countable discrete amenable group.

We say that $\mathcal{T}$ is a \textbf{tiling} of $G$ if $\mathcal{T}\subset\mathcal{F}(G)$, $\bigcup_{T\in\mathcal{T}}T=G$ and $T\cap T'=\emptyset$ holds for any two distinct $T,T'\in\mathcal{T}$. Every element in the tiling $\mathcal{T}$ is called a \textbf{$\mathcal{T}$-tile} (or a \textbf{tile}). A tiling $\mathcal{T}$ of $G$ is said to be \textbf{finite} if there is a finite collection $\mathcal{S}_{\mathcal{T}}\subset\mathcal{F}(G)$ such that every $\mathcal{T}$-tile is a translation of some element in $\mathcal{S}_{\mathcal{T}}$, i.e., for each $T\in\mathcal{T}$ there exist $S\in\mathcal{S}_{\mathcal{T}}$ and $c\in G$ such that $Sc=T$. Every element in $\mathcal{S}_{\mathcal{T}}$ is called a \textbf{shape} of $\mathcal{T}$. For every shape $S\in\mathcal{S}_{\mathcal{T}}$ the \textbf{center} of $S$ is defined by $$C(S)=\{c\in G:Sc\in\mathcal{T}\}\subset G.$$

The \textbf{translation} of a tiling $\mathcal{T}$ by $g\in G$ is $$\mathcal{T}g=\{Tg:T\in\mathcal{T}\},$$ which is also a tiling of $G$. For $F\in\mathcal{F}(G)$ we set $$\mathcal{T}|_{F}=\{T\cap F:T\in\mathcal{T}\}.$$ A finite tiling $\mathcal{T}$ of $G$ is called \textbf{syndetic} if for every shape $S\in\mathcal{S}_{\mathcal{T}}$ the center $C(S)$ is syndetic. A sequence $\{\mathcal{T}_{k}\}_{k=1}^{\infty}$ of finite tilings of $G$ is called \textbf{primely congruent} if for every $k\ge1$, $\mathcal{T}_{k}$ is a refinement of $\mathcal{T}_{k+1}$ (i.e. every $\mathcal{T}_{k+1}$-tile is a union of some $\mathcal{T}_{k}$-tiles) and each shape of $\mathcal{T}_{k+1}$ is partitioned by shapes of $\mathcal{T}_{k}$ in a unique way (i.e. for any two $\mathcal{T}_{k+1}$-tiles $Sc_{1}$ and $Sc_{2}$ of the same shape $S\in\mathcal{S}_{\mathcal{T}_{k+1}}$ we have ${\mathcal{T}_{k}}|_{Sc_{1}}=({\mathcal{T}_{k}}|_{Sc_{2}})c_{2}^{-1}c_{1}$).

\medskip

We list some propositions of tilings as follows, which are going to be used in our main proof. Some of these propositions may be found in \cite{Dou}. Here we reproduce their proofs for completeness.

\begin{proposition}\label{bigproportion}
Suppose that $\mathcal{T}$ is a finite tiling of $G$. Then for any $\epsilon>0$ there exist $K\in\mathcal{F}(G)$ and $\delta>0$ such that for each $g\in G$ and each $(K,\delta)$-invariant $F\in\mathcal{F}(G)$, the union of those $\mathcal{T}g$-tiles which are contained in $F$ has proportion larger than $1-\epsilon$, namely
$$\frac{|\bigcup_{T\in\mathcal{T}g,T\subset F}T|}{|F|}>1-\epsilon.$$
\end{proposition}
\begin{proof}
We assume that $\mathcal{S}_{\mathcal{T}}$ is a set of shapes of $\mathcal{T}$. Put
$$K=\bigcup_{S\in\mathcal{S}_{\mathcal{T}}}S.$$
For any $F\in\mathcal{F}(G)$ and $g\in G$,
$$\bigcup_{T\in\mathcal{T}g,T\subset F}T=\bigcup_{S\in\mathcal{S}_{\mathcal{T}},Sc\in\mathcal{T},Scg\subset F}Scg.$$
Set
$$\delta=\frac{\epsilon}{|K|}.$$
We claim that for any $(K,\delta)$-invariant $F\in\mathcal{F}(G)$,
$$F\setminus KB(F,K)\subset\bigcup_{T\in\mathcal{T}g,T\subset F}T.$$
In fact, if we take $h\in F\setminus KB(F,K)$, then by the definition of $B(F,K)$ we see that $KK^{-1}h\subset F$. Since $\mathcal{T}g$ is a tiling of $G$, we have $h\in Scg$ for some $S\in\mathcal{S}_{\mathcal{T}}$ and some $c\in C(S)$. This implies $cgh^{-1}\in S^{-1}$. It follows that $Scg=S(cgh^{-1})h\subset SS^{-1}h\subset KK^{-1}h\subset F$. So we get $h\in Scg\subset F$. Therefore $h\in\bigcup_{T\in\mathcal{T}g,T\subset F}T$. This proves our claim. Thus, by this claim we deduce
$$\Big|\bigcup_{T\in\mathcal{T}g,T\subset F}T\Big|\ge|F\setminus KB(F,K)|\ge|F|-|K|\cdot|B(F,K)|>(1-\epsilon)|F|.$$
\end{proof}

\begin{proposition}\label{manytiles}
Suppose that $\mathcal{T}$ is a syndetic finite tiling of $G$ and $\mathcal{S}_{\mathcal{T}}$ is a set of shapes of $\mathcal{T}$. Then for any $n\in\mathbb{N}$ there exist $K\in\mathcal{F}(G)$ and $\epsilon>0$ such that for every $S\in\mathcal{S}_{\mathcal{T}}$ and every $(K,\epsilon)$-invariant $F\in\mathcal{F}(G)$, $F$ contains at least $n$ $\mathcal{T}$-tiles of the shape $S$.
\end{proposition}
\begin{proof}
Without loss of generality, we assume that every shape $S\in\mathcal{S}_{\mathcal{T}}$ contains the identity element $e$ of $G$ (replacing $S$ by $Ss^{-1}$ for some $s\in S$ if necessary).

We claim that there exist $K^{\prime}\in\mathcal{F}(G)$ and $\epsilon^{\prime}>0$ such that every $(K^{\prime},\epsilon^{\prime})$-invariant finite subset of $G$ contains a $\mathcal{T}$-tile of the shape $S$ for each $S\in\mathcal{S}_{\mathcal{T}}$. In fact, since $\mathcal{S}_{\mathcal{T}}$ is a finite set, there exists $R\in\mathcal{F}(G)$ with $e\in R$, which does not depend on $S$, such that $RC(S)=G$, and therefore $RSC(S)=G$, for all $S\in\mathcal{S}_{\mathcal{T}}$. Set $T=\bigcup_{\mathcal{S}_{\mathcal{T}}}S$. Let $0<\epsilon^{\prime}<1$ and $K^{\prime}=RTT^{-1}R^{-1}$. Since $e\in K^{\prime}$, for any $(K^{\prime},\epsilon^{\prime})$-invariant $F\in\mathcal{F}(G)$ there is $g\in F$ with $K^{\prime}g\subset F$. Thus, for any $S\in\mathcal{S}_{\mathcal{T}}$ we have $g\in RSc$ for some $c\in C(S)$, and hence
$$Sc\subset RSc\subset RSS^{-1}R^{-1}g\subset K^{\prime}g\subset F.$$ This shows the claim.

Now let us fix $n\in\mathbb{N}$. We take $A\in\mathcal{F}(G)$ which is $(K^{\prime},\epsilon^{\prime})$-invariant and choose $g_1,g_2,\dots,g_n\in G$ such that $Ag_{1},Ag_{2},\dots,Ag_{n}$ are pairwise disjoint. Let $K=\bigcup_{j=1}^{n}Ag_{j}$ and $0<\epsilon<1$. We may assume that $K$ contains the identity element of $G$. Then for any $(K,\epsilon)$-invariant $F\in\mathcal{F}(G)$ there exists some $g\in F$ such that $Kg\subset F$, and hence $Ag_{j}g\subset F$ for all $1\le j\le n$. Since $A$ is $(K^{\prime},\epsilon^{\prime})$-invariant, we have for every $1\le j\le n$ that $Ag_{j}g$ is $(K^{\prime},\epsilon^{\prime})$-invariant as well, and hence contains a $\mathcal{T}$-tile of the shape $S$ for each $S\in\mathcal{S}_{\mathcal{T}}$. Thus, $F$ contains at least $n$ $\mathcal{T}$-tiles of the shape $S$ for every $S\in\mathcal{S}_{\mathcal{T}}$.
\end{proof}

\medskip

\subsection{Topological dimension and mean dimension}
Let $X$ be a compact metric space, $\rho$ a metric on $X$, and $P$ a polyhedron.
For $\epsilon>0$, a continuous mapping $f:X\to P$ is called an \textbf{$\epsilon$-embedding} with respect to $\rho$ if $f(x)=f(y)$ implies $\rho(x,y)<\epsilon$, for all $x,y\in X$. Let $\Widim_{\epsilon}(X,\rho)$ be the minimum dimension of a polyhedron $P$ such that there is an $\epsilon$-embedding $f:X\to P$. Recall that the \textbf{topological dimension} of $X$ may be recovered by $$\dim(X)=\lim_{\epsilon\to0}\Widim_{\epsilon}(X,\rho).$$

\medskip

Let $K$ be a compact metric space with a metric $d$. For every $n\in\mathbb{N}$ we equip the space $K^n$ with the product topology and define a compatible metric $d_{l^{\infty}}$ on $K^{n}$ by
\begin{equation}\label{infinitedistance}
d_{l^{\infty}}\left((x_{1},x_{2},\dots,x_{n}),(y_{1},y_{2},\dots,y_{n})\right)=\max_{1\le i\le n}d(x_{i},y_{i}).
\end{equation}
We include here a practical theorem.
\begin{theorem}[{\cite[Lemma 3.2]{Lindenstrauss--Weiss}}]\label{widimcube}
For any $0<\epsilon<1$ and any $n\in\mathbb{N}$ we have
$$\Widim_{\epsilon}\left([0,1]^{n},d_{l^{\infty}}\right)=n.$$
In particular, $\dim\left([0,1]^{n}\right)=n$.
\end{theorem}

\medskip

Let $(X,G)$ be a $G$-action and $d$ a metric on $X$. For $F\in\mathcal{F}(G)$ and $x,y\in X$ we set
$$d_{F}(x,y)=\max_{g\in F}d(gx,gy).$$
The \textbf{mean dimension} of $(X,G)$ is defined by
$$\mdim(X,G)=\lim_{\epsilon\to0}\lim_{n\to\infty}\frac{\Widim_{\epsilon}(X,d_{F_{n}})}{|F_{n}|},$$
where $\{F_{n}\}_{n=1}^{\infty}$ is a F\o lner sequence of $G$.
It is well known that the limit in the above definition always exists\footnote{The existence of the inner limit is due to the Ornstein--Weiss theorem (see \cite[Theorem 6.1]{Lindenstrauss--Weiss}). The outer limit exists because $\Widim_{\epsilon}(X,d_{F_{n}})$ is monotone with respect to $\epsilon$.}, and the value $\mdim(X,G)$ is independent of the choice of a F\o lner sequence of $G$.

\bigskip

\section{A constructive proof of Theorem \ref{maintheorem}}\label{sec:maintheorem}
The proof of Theorem \ref{maintheorem} consists of four parts. Part 1 is dedicated to the construction, while Parts 2,3,4 are devoted to the argument that the $G$-action we constructed satisfies all the required conditions. The title of each part indicates the precise aim of the part.

\medskip

Let us start with necessary settings. We denote by $D_3$ the discrete space consisting of three points and by $P$ the cone of $D_3$, namely,
\begin{equation*}
P=([0,1]\times D_3)/\sim,
\end{equation*}
where $(0,a)\sim(0,b)$ for all $a,b\in D_3$. Obviously, $\dim(P)=1$. Throughout this section, we let $d$ be the graph distance on $P$ with all three edges having length one and $d_{l^\infty}$ the metric on $P^n$ defined by \eqref{infinitedistance} for $n\in\mathbb{N}$. We include a topological embedding result as follows.
\begin{theorem}[{\cite[Proposition 2.5]{Lindenstrauss--Tsukamoto}}]\label{epsilonembedding}
For every $\epsilon\in(0,1)$, there does not exist an $\epsilon$-embedding of $(P^n,d_{l^\infty})$ into $\mathbb{R}^{2n-1}$ for any $n\in\mathbb{N}$.
\end{theorem}
We make use of a recent result on tilings of amenable groups.
\begin{theorem}[{\cite[Theorem 5.2]{DHZ}},{\cite[Theorem 3.6]{Dou}}]\label{tiling}
Let $G$ be an infinite countable amenable group with the identity element $e$, $\{T_k\}_{k=1}^\infty\subset\mathcal{F}(G)$ an increasing sequence with $\bigcup_{k=1}^{\infty}T_k=G$, and $\{\epsilon_k\}_{k=1}^\infty$ a decreasing sequence of positive numbers converging to zero. Then there exists a primely congruent sequence $\{\mathcal{T}_k\}_{k=1}^\infty$ of syndetic\footnote{The term ``syndetic'' here corresponds to the term ``irreducible'' in \cite{Dou}.} finite tilings of $G$ satisfying the following conditions:
\begin{enumerate}
\item $e\in S_{1,1}\subset S_{2,1}\subset\cdots\subset S_{k,1}\subset\cdots\subset\bigcup_{k=1}^{\infty}S_{k,1}=G$;
\item for every $k\in\mathbb{N}$ and every $1\le i\le m_k$, $S_{k,i}$ is $(T_k,\epsilon_k)$-invariant;
\end{enumerate}
where for each $k\in\mathbb{N}$, $\{S_{k,i}:1\le i\le m_k\}$ is the set of all shapes of $\mathcal{T}_k$.
\end{theorem}

\medskip

Let $G=\{g_k:k\in\mathbb{N}\}$ be an infinite countable discrete amenable group whose identity element is denoted by $e$. Take a decreasing sequence $\{\eta_n\}_{n=1}^\infty$ of positive numbers converging to zero and an increasing sequence $\{A_n\}_{n=1}^\infty\subset\mathcal{F}(G)$ with $\bigcup_{n=1}^{\infty}A_n=G$. By Theorem \ref{tiling}, there exists a primely congruent sequence $\{\mathcal{T}_n\}_{n=1}^\infty$ of syndetic finite tilings of $G$ with the sets of shapes $\mathcal{S}_{\mathcal{T}_n}=\{S_{n,i}:1\le i\le m_n\}$ satisfying that
$$e\in S_{1,1}\subset S_{2,1}\subset\cdots\subset S_{n,1}\subset\cdots\subset\bigcup_{n=1}^{\infty}S_{n,1}=G,\quad e\in C(S_{n,1}),$$
and that $S_{n,i}$ is $(A_n,\eta_n)$-invariant for every $n\in\mathbb{N}$ and every $1\le i\le m_n$.

\medskip

Without loss of generality, we may assume $d=1$ in the statement of Theorem \ref{maintheorem} (otherwise, we replace $P$ by $P^d$ in our argument). We are going to construct a required $G$-action, which is a subshift of the full $G$-shift on $P^G$. We denote it by $(X,\sigma)$.

\medskip

Let $\rho$ and $\rho^\prime$ be the metrics on $P^G$ and $[0,1]^G$, respectively, defined by \eqref{metricrho}. Let $\{\delta_n\}_{n=1}^\infty$ be a strictly decreasing sequence of positive numbers converging to zero. Take an increasing sequence $\{P_n\}_{n=1}^\infty$ of finite subsets of $P$ such that for each $n\in\mathbb{N}$, $P_n$ is $\delta_n$-dense in $P$. Let $\{F_n\}_{n=1}^\infty\subset\mathcal{F}(G)$ be an increasing sequence with $\bigcup_{n=1}^{\infty}F_n=G$.\footnote{Note that $\{F_n\}_{n=1}^\infty$ will play a role in the proof different from $\{A_n\}_{n=1}^\infty$, although $\{F_n\}_{n=1}^\infty$ could be, of course, the same as $\{A_n\}_{n=1}^\infty$.} We take a symbol $\ast\notin P$ and set $\hat{P}=P\cup\{\ast\}$.

\bigskip

\noindent\textbf{Part 1: Construction of $(X,\sigma)$.}
The construction of $(X,\sigma)$ will be completed by induction.

\medskip

\textbf{Step 1.}
We choose $n_1\in\mathbb{N}$ sufficiently large so that for every $1\le i\le m_{n_1}$ there is $x_{1,i}\in\hat{P}^{S_{n_1,i}}$ with
$$\frac{1+\delta_1}{2}<\frac{|x_{1,i}(S_{n_1,i},\ast)|}{|S_{n_1,i}|}\le\frac{1+\delta_1}{2}+\frac{1}{|S_{n_1,i}|}.$$
Let
$$B_{1,i}=\left\{x=(x_g)_{g\in S_{n_1,i}}\in P^{S_{n_1,i}}:x_g=(x_{1,i})_g,\;\forall g\in S_{n_1,i}\setminus x_{1,i}(S_{n_1,i},\ast)\right\}.$$
We define $x_1\in\hat{P}^G$ by
$${x_1}|_{S_{n_1,i}c}=x_{1,i},\quad\forall1\le i\le m_{n_1},\;\forall c\in C(S_{n_1,i}).$$
We set
$$X_1=\left\{x\in P^G:x|_{S_{n_1,i}c}\in B_{1,i},\;\forall1\le i\le m_{n_1},\forall c\in C(S_{n_1,i})\right\}.$$

\medskip

\textbf{Step 2.}
Applying Proposition \ref{manytiles}, we choose $l_1\in\mathbb{N}$ sufficiently large such that we can find a finite subset $R_1\subset C(S_{n_1,1})$ and $h_1\in C(S_{n_1,1})$ satisfying
$$e\in R_1,\quad h_1\notin R_1,\quad S_{n_1,1}R_1\cup S_{n_1,1}h_1\subset S_{l_1,1},\quad|R_1|=|P_1|^{|x_{1,1}(S_{n_1,1},\ast)|}.$$
We select $w_1\in\hat{P}^{S_{l_1,1}}$ such that Conditions (A.2.1), (A.2.2), (A.2.3) are satisfied:
\begin{enumerate}
\item[(A.2.1)]
${w_1}|_{S_{n_1,1}r\setminus{x_{1,1}(S_{n_1,1},\ast)r}}={x_{1,1}}|_{S_{n_1,1}\setminus{x_{1,1}(S_{n_1,1},\ast)}}$, $\forall r\in R_1$;
\item[(A.2.2)]
${w_1}|_{x_{1,1}(S_{n_1,1},\ast)r}\in P_1^{|x_{1,1}(S_{n_1,1},\ast)|}$ $(r\in R_1)$ are pairwise distinct, i.e.\footnote{Precisely speaking, here (as well as in (A.k.2)) when we compare two ``vectors'', say, ${w_1}|_{x_{1,1}(S_{n_1,1},\ast)}$ and ${w_1}|_{x_{1,1}(S_{n_1,1},\ast)r}$, we agree that their ``coordinates'' correspond synchronously under the right-multiplication with $r$ taken within $R_1$, i.e. ${w_1}|_{x_{1,1}(S_{n_1,1},\ast)}={w_1}|_{x_{1,1}(S_{n_1,1},\ast)r}$ if and only if $(w_1)_g=(w_1)_{gr}\in P_1$ for all $g\in{x_{1,1}(S_{n_1,1},\ast)}$.}
$$\left\{{w_1}|_{x_{1,1}(S_{n_1,1},\ast)r}:r\in R_1\right\}=P_1^{|x_{1,1}(S_{n_1,1},\ast)|};$$
\item[(A.2.3)]
if $S_{n_1,i}c\subset S_{l_1,1}\setminus S_{n_1,1}R_1$ for some $1\le i\le m_{n_1}$ and some $c\in C(S_{n_1,i})$ then
$${w_1}|_{S_{n_1,i}c}=x_{1,i}.$$
\end{enumerate}
Clearly,
$${w_1}|_{S_{n_1,1}h_1}=x_{1,1},\quad{w_1}|_{S_{n_1,1}}\in B_{1,1}\subset P^{S_{n_1,1}}.$$
We pick $n_2\in\mathbb{N}$ sufficiently large such that Conditions (B.2.1), (B.2.2), (B.2.3), (B.2.4) are satisfied:
\begin{enumerate}
\item[(B.2.1)]
$g_1h_1\in S_{n_2,1}$;
\item[(B.2.2)]
$|F_1S_{n_2,1}|<(1+\delta_2)\cdot|S_{n_2,1}|$;
\item[(B.2.3)]
for every $1\le i\le m_{n_2}$, there is $c_{1,i}\in C(S_{l_1,1})$ such that $S_{l_1,1}c_{1,i}\subset S_{n_2,i}$, and moreover, $c_{1,1}=e$;
\item[(B.2.4)]
for every $1\le i\le m_{n_2}$, $|S_{l_1,1}|$ is negligible compared with $|S_{n_2,i}|$, more precisely,
$$\frac{|x_1(S_{n_2,i}\setminus S_{l_1,1}c_{1,i},\ast)|}{|S_{n_2,i}|}>\frac{1+\delta_1}{2},
\quad\forall1\le i\le m_{n_2}.$$
\end{enumerate}
For every $1\le i\le m_{n_2}$ we choose $x_{2,i}\in\hat{P}^{S_{n_2,i}}$ such that Conditions (C.2.1), (C.2.2), (C.2.3) are satisfied:
\begin{enumerate}
\item[(C.2.1)]
${x_{2,i}}|_{S_{l_1,1}c_{1,i}}=w_1$;
\item[(C.2.2)]
if $S_{n_1,j}c\subset S_{n_2,i}\setminus S_{l_1,1}c_{1,i}$ for some $1\le j\le m_{n_1}$ and some $c\in C(S_{n_1,j})$ then
$$(x_{2,i})_{gc}=(x_{1,j})_g,\quad\forall g\in S_{n_1,j}\setminus x_{1,j}(S_{n_1,j},\ast);$$
\item[(C.2.3)]
on the rest of coordinates in $S_{n_2,i}\setminus S_{l_1,1}c_{1,i}$, there are appropriately many $\ast$'s such that
$$\frac{1+\delta_2}{2}<\frac{|x_{2,i}(S_{n_2,i},\ast)|}{|S_{n_2,i}|}\le\frac{1+\delta_2}{2}+\frac{1}{|S_{n_2,i}|}.$$
\end{enumerate}
Let
$$B_{2,i}=\left\{x=(x_g)_{g\in S_{n_2,i}}\in P^{S_{n_2,i}}:x_g=(x_{2,i})_g,\;\forall g\in S_{n_2,i}\setminus x_{2,i}(S_{n_2,i},\ast)\right\}.$$
We define $x_2\in\hat{P}^G$ by
$${x_2}|_{S_{n_2,i}c}=x_{2,i},\quad\forall1\le i\le m_{n_2},\;\forall c\in C(S_{n_2,i}).$$
We set
$$X_2=\left\{x\in P^G:x|_{S_{n_2,i}c}\in B_{2,i},\;\forall1\le i\le m_{n_2},\forall c\in C(S_{n_2,i})\right\}.$$

\bigskip

To proceed, we assume that $x_{k-1,i}$, $B_{k-1,i}$ ($1\le i\le m_{n_{k-1}}$), $x_{k-1}$ and $X_{k-1}$ have been already generated in Step $(k-1)$. Now we generate $x_{k,i}$, $B_{k,i}$ ($1\le i\le m_{n_k}$), $x_k$ and $X_k$ in Step $k$ ($k\ge2$).

\medskip

\textbf{Step k.}
By Proposition \ref{manytiles}, we take $l_{k-1}\in\mathbb{N}$ large enough such that we can find a finite subset $R_{k-1}\subset C(S_{n_{k-1},1})$ and $h_{k-1}\in C(S_{n_{k-1},1})$ satisfying
$$e\in R_{k-1},\quad h_{k-1}\notin R_{k-1},\quad S_{n_{k-1},1}R_{k-1}\cup S_{n_{k-1},1}h_{k-1}\subset S_{l_{k-1},1},$$
$$|R_{k-1}|=|P_{k-1}|^{|x_{k-1,1}(S_{n_{k-1},1},\ast)|}.$$
We select $w_{k-1}\in\hat{P}^{S_{l_{k-1},1}}$ such that Conditions (A.k.1), (A.k.2), (A.k.3) are satisfied:
\begin{enumerate}
\item[(A.k.1)]
${w_{k-1}}|_{S_{n_{k-1},1}r\setminus{x_{k-1,1}(S_{n_{k-1},1},\ast)r}}={x_{k-1,1}}|_{S_{n_{k-1},1}\setminus{x_{k-1,1}(S_{n_{k-1},1},\ast)}}$, $\forall r\in R_{k-1}$;
\item[(A.k.2)]
${w_{k-1}}|_{x_{k-1,1}(S_{n_{k-1},1},\ast)r}\in P_{k-1}^{|x_{k-1,1}(S_{n_{k-1},1},\ast)|}$ $(r\in R_{k-1})$ are pairwise distinct, i.e.
$$\left\{{w_{k-1}}|_{x_{k-1,1}(S_{n_{k-1},1},\ast)r}:r\in R_{k-1}\right\}=P_{k-1}^{|x_{k-1,1}(S_{n_{k-1},1},\ast)|};$$
\item[(A.k.3)]
if $S_{n_{k-1},i}c\subset S_{l_{k-1},1}\setminus S_{n_{k-1},1}R_{k-1}$ for some $1\le i\le m_{n_{k-1}}$ and some $c\in C(S_{n_{k-1},i})$ then
$${w_{k-1}}|_{S_{n_{k-1},i}c}=x_{k-1,i}.$$
\end{enumerate}
Obviously,
$${w_{k-1}}|_{S_{n_{k-1},1}h_{k-1}}=x_{k-1,1},\quad{w_{k-1}}|_{S_{n_{k-1},1}}\in B_{k-1,1}\subset P^{S_{n_{k-1},1}}.$$
We pick $n_k\in\mathbb{N}$ sufficiently large such that Conditions (B.k.1), (B.k.2), (B.k.3), (B.k.4) are satisfied:
\begin{enumerate}
\item[(B.k.1)]
$g_{k-1}h_1h_2\cdots h_{k-1}\in S_{n_k,1}$;
\item[(B.k.2)]
$|F_{k-1}S_{n_k,1}|<(1+\delta_k)\cdot|S_{n_k,1}|$;
\item[(B.k.3)]
for every $1\le i\le m_{n_k}$, there is $c_{k-1,i}\in C(S_{l_{k-1},1})$ such that $S_{l_{k-1},1}c_{k-1,i}\subset S_{n_k,i}$, and moreover, $c_{k-1,1}=e$;
\item[(B.k.4)]
for every $1\le i\le m_{n_k}$, $|S_{l_{k-1},1}|$ is negligible compared with $|S_{n_k,i}|$, more precisely,
$$\frac{|x_{k-1}(S_{n_k,i}\setminus S_{l_{k-1},1}c_{k-1,i},\ast)|}{|S_{n_k,i}|}>\frac{1+\delta_{k-1}}{2},
\quad\forall1\le i\le m_{n_k}.$$
\end{enumerate}
For every $1\le i\le m_{n_k}$ we choose $x_{k,i}\in\hat{P}^{S_{n_k,i}}$ such that Conditions (C.k.1), (C.k.2), (C.k.3) are satisfied:
\begin{enumerate}
\item[(C.k.1)]
${x_{k,i}}|_{S_{l_{k-1},1}c_{k-1,i}}=w_{k-1}$;
\item[(C.k.2)]
if $S_{n_{k-1},j}c\subset S_{n_k,i}\setminus S_{l_{k-1},1}c_{k-1,i}$ for some $1\le j\le m_{n_{k-1}}$ and some $c\in C(S_{n_{k-1},j})$ then
$$(x_{k,i})_{gc}=(x_{k-1,j})_g,\quad\forall g\in S_{n_{k-1},j}\setminus x_{k-1,j}(S_{n_{k-1},j},\ast);$$
\item[(C.k.3)]
on the rest of coordinates in $S_{n_k,i}\setminus S_{l_{k-1},1}c_{k-1,i}$, there are appropriately many $\ast$'s such that
$$\frac{1+\delta_k}{2}<\frac{|x_{k,i}(S_{n_k,i},\ast)|}{|S_{n_k,i}|}\le\frac{1+\delta_k}{2}+\frac{1}{|S_{n_k,i}|}.$$
\end{enumerate}
Let
$$B_{k,i}=\left\{x=(x_g)_{g\in S_{n_k,i}}\in P^{S_{n_k,i}}:x_g=(x_{k,i})_g,\;\forall g\in S_{n_k,i}\setminus x_{k,i}(S_{n_k,i},\ast)\right\}.$$
We define $x_k\in\hat{P}^G$ by
$${x_k}|_{S_{n_k,i}c}=x_{k,i},\quad\forall1\le i\le m_{n_k},\;\forall c\in C(S_{n_k,i}).$$
We set
$$X_k=\left\{x\in P^G:x|_{S_{n_k,i}c}\in B_{k,i},\;\forall1\le i\le m_{n_k},\forall c\in C(S_{n_k,i})\right\}.$$

\bigskip

So far we have already generated $x_{k,i}$, $B_{k,i}$ ($1\le i\le m_{n_k}$), $x_k$ and $X_k$ in Step $k$ for all $k\in\mathbb{N}$. It follows from our construction that $\{X_k\}_{k=1}^\infty$ is a decreasing sequence of nonempty subsets of $P^G$, and
$$x_{k+1}|_{S_{n_k,1}}=x_m|_{S_{n_k,1}},\quad\forall k\in\mathbb{N},\;\forall m\ge k+1.$$
Now by the fact $\bigcup_{k=1}^{\infty}S_{n_k,1}=G$ we observe that if a point $x$ belongs to the intersection $\bigcap_{k=1}^{\infty}X_k$ then the value $x_g\in P$ ($g\in G$) for all its coordinates must be determined eventually according to our construction.
Thus, the intersection $\bigcap_{k=1}^{\infty}X_k$ contains in fact only one point. We set
$$\bigcap_{k=1}^{\infty}X_k=\{z\}.$$
Finally, we let $X\subset P^G$ be the orbit closure of $z$, i.e.
$$X=\overline{Gz}=\overline{\{gz:g\in G\}}.$$
Since $X$ is a closed subset of $P^G$ and is invariant under the $G$-shift, $(X,\sigma)$ becomes a subshift of $(P^G,\sigma)$. This eventually finishes the construction of $(X,\sigma)$. Now we check that $(X,\sigma)$ satisfies all the required properties.

\bigskip

\noindent\textbf{Part 2: Minimality of $(X,\sigma)$.}
To show that $(X,\sigma)$ is minimal, it suffices to prove that the point $z\in X$ is almost periodic, i.e. for any $\epsilon>0$ there exists a syndetic subset $S=S_\epsilon$ of $G$ with
$$\rho(z,cz)<\epsilon,\quad\forall c\in S.$$

\medskip

To see the latter statement, we fix $\epsilon>0$ arbitrarily. Since $S_{k,1}$ is increasing over $k\in\mathbb{N}$ and eventually covers the group $G$, there exists $m\in\mathbb{N}$ such that
$$x|_{S_{n_m,1}}=x^\prime|_{S_{n_m,1}}\quad\text{implies}\quad\rho(x,x^\prime)<\epsilon.$$
Since the tiling $\mathcal{T}_{n_{m+1}}$ is syndetic, $C(S_{n_{m+1},1})$ is syndetic. By the definition of $z$ in the construction, we have
$$z|_{S_{n_m,1}}=z|_{S_{n_m,1}c},\quad\forall c\in C(S_{n_{m+1},1})$$
i.e.
$$z|_{S_{n_m,1}}=(cz)|_{S_{n_m,1}},\quad\forall c\in C(S_{n_{m+1},1}).$$
It follows that
$$\rho(z,cz)<\epsilon,\quad\forall c\in C(S_{n_{m+1},1}).$$
Thus, we end this part with taking $S=C(S_{n_{m+1},1})$.

\bigskip

\noindent\textbf{Part 3: Mean dimension of $(X,\sigma)$.}
The aim of this part is to prove
$$\mdim(X,\sigma)=\frac12.$$

\medskip

The following well-known proposition is a useful tool for an upper bound of mean dimension of subshifts. We reproduce its proof for completeness.
\begin{proposition}\label{upperbound}
Let $K$ be a finite dimensional compact metric space and $(X,\sigma)$ a subshift of $(K^{G},\sigma)$. Then
$$\mdim(X,\sigma)\le\liminf_{n\to\infty}\frac{\dim(\pi_{F_{n}}(X))}{|F_{n}|}$$
for any F\o lner sequence $\{F_{n}\}_{n=1}^{\infty}$ of $G$.
\end{proposition}
\begin{proof}
Let $\rho$ be the metric on $K^{G}$ defined by \eqref{metricrho}. Fix a F\o lner sequence $\{F_{n}\}_{n=1}^{\infty}$ of $G$ and take $\epsilon,\delta>0$. We choose $A\in\mathcal{F}(G)$ containing the identity element of $G$ such that if two points $x,y\in K^{G}$ satisfy $x|_{A}=y|_{A}$ then $\rho(x,y)<\epsilon$. It follows that for any $x,y\in K^{G}$ and $n\in\mathbb{N}$, if $x|_{AF_{n}}=y|_{AF_{n}}$ (i.e. $\pi_{AF_{n}}(x)=\pi_{AF_{n}}(y)$) then $\rho_{F_{n}}(x,y)<\epsilon$. Thus, for any $n\in\mathbb{N}$, $$(\pi_{AF_{n}})|_X:X\to\pi_{AF_{n}}(X)$$ is an $\epsilon$-embedding with respect to the metric $\rho_{F_{n}}$, and therefore $$\Widim_{\epsilon}(X,\rho_{F_{n}})\le\dim(\pi_{AF_{n}}(X)).$$
By noting that $$\pi_{AF_{n}}(X)\subset\pi_{F_{n}}(X)\times K^{AF_{n}\setminus F_{n}}$$ we have
$$\Widim_{\epsilon}(X,\rho_{F_{n}})\le\dim(\pi_{AF_{n}}(X))\le\dim(\pi_{F_{n}}(X))+|AF_{n}\setminus F_{n}|\cdot\dim(K)$$
for all $n\in\mathbb{N}$.

Take a sufficiently large $N\in\mathbb{N}$ such that
$$\frac{|AF_{n}\setminus F_{n}|}{|F_{n}|}<\frac{\delta}{\dim(K)+1},$$
for all $n\ge N$. Then
\begin{align*}
\lim_{n\to\infty}\frac{\Widim_{\epsilon}(X,\rho_{F_{n}})}{|F_{n}|}
&\le\liminf_{n\to\infty}\frac{\dim(\pi_{F_{n}}(X))}{|F_{n}|}+\limsup_{n\to\infty}\frac{|AF_{n}\setminus F_{n}|}{|F_{n}|}\dim(K)\\
&\le\liminf_{n\to\infty}\frac{\dim(\pi_{F_{n}}(X))}{|F_{n}|}+\delta.
\end{align*}
Since $\delta>0$ is arbitrary,
$$\lim_{n\to\infty}\frac{\Widim_{\epsilon}(X,\rho_{F_{n}})}{|F_{n}|}\le\liminf_{n\to\infty}\frac{\dim(\pi_{F_{n}}(X))}{|F_{n}|}.$$
Letting $\epsilon\to0$, we end the proof.
\end{proof}

\medskip

To estimate $\mdim(X,\sigma)$ from above, we fix $k\in\mathbb{N}$ and $\epsilon>0$ arbitrarily. We denote by $\widetilde{X_k}$ the subshift of $P^G$ generated by $X_k$, namely,
$$\widetilde{X_k}=\overline{GX_k}=\overline{\bigcup_{g\in G}gX_k}.$$
Take a F\o lner sequence $\{E_n\}_{n=1}^\infty$ of $G$. By Proposition \ref{bigproportion}, there exists $N_0\in\mathbb{N}$ sufficiently large such that for any $n\ge N_0$, the union of $\mathcal{T}_{n_k}g$-tiles which are contained in $E_n$ has proportion larger than $1-\epsilon$ for all $g\in G$, i.e.
$$\frac{|\bigcup_{T\in\mathcal{T}_{n_k}g,T\subset E_n}T|}{|E_n|}>1-\epsilon,\quad\forall n\ge N_0,\;\forall g\in G.$$
For any $n\ge N_0$ we divide $G$ into $L_{n,k}$ classes $Q_1,Q_2,\dots,Q_{L_{n,k}}$ such that if $g,h\in Q_i$ for some $1\le i\le L_{n,k}$ then
$$\mathcal{T}_{n_k}g|_{E_n}=\mathcal{T}_{n_k}h|_{E_n},$$
where $\mathcal{T}_{n_k}g|_{E_n}=\{Tg\cap E_n:T\in\mathcal{T}_{n_k}\}$. Since $E_n$ is finite, $L_{n,k}$ is a finite number. For each $1\le i\le L_{n,k}$ we take $q_i\in Q_i$. For every $n\ge N_0$ and every $1\le i\le L_{n,k}$, there is $j_{n,i}\in\mathbb{N}$ such that
$$\mathcal{T}_{n_k}q_i|_{E_n}=\left\{S_{n_k,p_{n,1}}c_{n,1}q_i,S_{n_k,p_{n,2}}c_{n,2}q_i,\dots,S_{n_k,p_{n,j_{n,i}}}c_{n,j_{n,i}}q_i,\;A_{n,i}\right\}$$
for some $1\le p_{n,l}\le m_{n_k}$, $c_{n,l}\in C(S_{n_k,p_{n,l}})$ ($1\le l\le j_{n,i}$) and some $A_{n,i}\subset E_n$ with $|A_{n,i}|/|E_n|<\epsilon$. By the construction of $\widetilde{X_k}$,
$$\pi_{E_n}(\widetilde{X_k})\subset\bigcup_{1\le i\le L_{n,k}}B_{k,p_{n,1}}\times B_{k,p_{n,2}}\times\cdots\times B_{k,p_{n,j_{n,i}}}\times P^{A_{n,i}},\quad\forall n\ge N_0.$$
Thus, we have
\begin{align*}
\frac{\dim(\pi_{E_n}(\widetilde{X_k}))}{|E_n|}
&\le\max_{1\le i\le L_{n,k}}\frac{\dim(B_{k,p_{n,1}}\times B_{k,p_{n,2}}\times\cdots\times B_{k,p_{n,j_{n,i}}}\times P^{A_{n,i}})}{|E_n|}\\
&\le\max_{1\le i\le L_{n,k}}\frac{\sum_{1\le l\le j_{n,i}}\dim(B_{k,p_{n,l}})+|A_{n,i}|}{|E_n|}\\
&\le\max_{1\le i\le L_{n,k}}\frac{\sum_{1\le l\le j_{n,i}}\left((1+\delta_k)/2+1/|S_{n_k,p_{n,l}}|\right)\cdot|S_{n_k,p_{n,l}}|}{|E_n|}+\epsilon\\
&<\frac{1+\delta_k}{2}+\frac{1}{\min\{|S_{n_k,j}|:1\le j\le m_{n_k}\}}+\epsilon
\end{align*}
for all $n\ge N_0$. By Proposition \ref{upperbound}, we obtain
$$\mdim(\widetilde{X_k},\sigma)\le\frac{1+\delta_k}{2}+\frac{1}{\min\{|S_{n_k,j}|:1\le j\le m_{n_k}\}}+\epsilon.$$
Since $k\in\mathbb{N}$ and $\epsilon>0$ are arbitrary, and since $$\mdim(X,\sigma)\le\mdim(\widetilde{X_k},\sigma)$$ for all $k\in\mathbb{N}$, it follows that
$$\mdim(X,\sigma)\le\lim_{k\to\infty}\left(\frac{1+\delta_k}{2}+\frac{1}{\min\{|S_{n_k,j}|:1\le j\le m_{n_k}\}}\right)=\frac12.$$

\bigskip

In order to show
$$\mdim(X,\sigma)\ge\frac12,$$
we need more preparations. Set
$$T_1=S_{n_1,1},\quad T_k=S_{n_k,1}h_{k-1}^{-1}\cdots h_1^{-1},\quad\forall k\ge2.$$
Since $\{S_{n_k,1}\}_{k=1}^\infty$ is a F\o lner sequence of $G$, so is the sequence $\{T_k\}_{k=1}^\infty$. According to the choice of $l_k$, we have
$$S_{n_k,1}\subset S_{l_k,1}h_k^{-1}\subset S_{n_{k+1},1}h_k^{-1},\quad\forall k\in\mathbb{N}.$$
It follows that
$$T_k=S_{n_k,1}h_{k-1}^{-1}\cdots h_1^{-1}\subset S_{n_{k+1},1}h_k^{-1}h_{k-1}^{-1}\cdots h_1^{-1}=T_{k+1},\quad\forall k\in\mathbb{N}.$$
By (B.k.1),
$$g_k\in S_{n_{k+1},1}h_k^{-1}\cdots h_1^{-1}=T_{k+1},\quad\forall k\in\mathbb{N}.$$
Therefore $\{T_k\}_{k=1}^\infty$ is an increasing F\o lner sequence of $G$ with $$\bigcup_{k=1}^{\infty}T_k=G.$$
Set
$$J_1=\{g\in S_{n_1,1}:(x_{1,1})_g=\ast\},$$
$$J_k=\{g\in S_{n_k,1}:(x_{k,1})_g=\ast\}h_{k-1}^{-1}\cdots h_1^{-1},\quad\forall k\ge2.$$
It follows from (A.k.3), (B.k.3), (C.k.1) that
$$\left\{g\in S_{n_k,1}:(x_{k,1})_g=\ast\right\}h_k\subset\left\{g\in S_{n_{k+1},1}:(x_{k+1,1})_g=\ast\right\},\quad\forall k\in\mathbb{N}.$$
Thus,
$$J_k\subset J_{k+1},\quad\forall k\in\mathbb{N}.$$
Let
$$J=\bigcup_{k=1}^{\infty}J_k.$$

\medskip

\begin{lemma}\label{projection}
For any $u,v\in P^J$ we can find $x,y\in X$ such that
$$x|_J=u,\quad y|_J=v,\quad x|_{G\setminus J}=y|_{G\setminus J}.$$
\end{lemma}
\begin{proof}
Take $u=(u_g)_{g\in J}\in P^J$. For every $k\in\mathbb{N}$ we define $\overline{u}_k\in B_{k,1}\subset P^{S_{n_k,1}}$ by
$$(\overline{u}_k)_g=\begin{cases}
(x_{k,1})_g,&\quad g\in S_{n_k,1}\setminus J_kh_1\cdots h_{k-1},\\
u_{gh_{k-1}^{-1}\cdots h_1^{-1}},&\quad g\in J_kh_1\cdots h_{k-1}.
\end{cases}$$
For each $m\in\mathbb{N}$ we take $u^m=(u_g^m)_{g\in J}\in P_m^J$ such that
$$\lim_{m\to\infty}u_g^m=u_g,\footnote{We mean $\lim_{m\to\infty}d(u_g^m,u_g)=0$.}\quad\forall g\in J.$$
For every $k\in\mathbb{N}$ and every $m\in\mathbb{N}$ we define $\overline{u}_k^m\in B_{k,1}\subset P^{S_{n_k,1}}$ by
$$(\overline{u}_k^m)_g=\begin{cases}
(x_{k,1})_g,&\quad g\in S_{n_k,1}\setminus J_kh_1\cdots h_{k-1},\\
u_{gh_{k-1}^{-1}\cdots h_1^{-1}}^m,&\quad g\in J_kh_1\cdots h_{k-1}.
\end{cases}$$
Clearly,
$$\lim_{m\to\infty}\overline{u}_k^m=\overline{u}_k,\quad\forall k\in\mathbb{N}.$$
Since
$$x_{k,1}=x_{k+1,1}|_{S_{n_k,1}h_k}=x_{k+2,1}|_{S_{n_k,1}h_kh_{k+1}}=\cdots=x_{m,1}|_{S_{n_k,1}h_kh_{k+1}\cdots h_{m-1}},\quad\forall m>k\ge1,$$
we have
$$x_{k,1}(S_{n_k,1},\ast)h_kh_{k+1}\cdots h_{m-1}\subset x_{m,1}(S_{n_m,1},\ast),\quad\forall m>k\ge1.$$
It follows that
$$P_k^{x_{k,1}(S_{n_k,1},\ast)}\subset\left\{{x_{m+1,1}}|_{x_{k,1}(S_{n_k,1},\ast)h_kh_{k+1}\cdots h_{m-1}r}:r\in R_m\right\},\quad\forall m>k\ge1.$$
Thus, for every $m>k\ge1$ there is some $r_{k,m}\in R_m$ such that
\begin{align*}
\overline{u}_k^m
&={x_{m+1,1}}|_{S_{n_k,1}h_k\cdots h_{m-1}r_{k,m}}\\
&=z|_{S_{n_k,1}h_k\cdots h_{m-1}r_{k,m}}\\
&=(h_k\cdots h_{m-1}r_{k,m}z)|_{S_{n_k,1}}.
\end{align*}
Notice that for any $k\in\mathbb{N}$, the limit of the sequence $\{h_k\cdots h_{m-1}r_{k,m}z\}_{m=1}^\infty$ exists. We assume
$$\lim_{m\to\infty}h_k\cdots h_{m-1}r_{k,m}z=z^\prime_k,\quad\forall k\in\mathbb{N}.$$
Clearly,
$$z^\prime_k\in X,\quad\overline{u}_k={z^\prime_k}|_{S_{n_k,1}},\quad\forall k\in\mathbb{N}.$$
Note that
$$(h_1\cdots h_{k-1}z^\prime_k)|_{J_k}={z^\prime_k}|_{J_kh_1\cdots h_{k-1}}={\overline{u}_k}|_{J_kh_1\cdots h_{k-1}}=u|_{J_k},\quad\forall k\in\mathbb{N}.$$
Let
$$x=\lim_{k\to\infty}h_1\cdots h_{k-1}z^\prime_k\in X.$$
We have
$$x|_J=u.$$
For the moment let us take an arbitrary $g\in G\setminus J$. We recall here that $S_{n_k,1}h_{k-1}^{-1}\cdots h_1^{-1}=T_k$ and $J_k$ are increasing over $k\in\mathbb{N}$, and eventually cover $G$ and $J$, respectively. So there is some $l(g)\in\mathbb{N}$ satisfying
$$g\in S_{n_k,1}h_{k-1}^{-1}\cdots h_1^{-1}\setminus J_k,\quad\forall k\ge l(g).$$
Since for every $k\ge l(g)$ it holds that
\begin{align*}
(h_1\cdots h_{k-1}z^\prime_k)|_{S_{n_k,1}h_{k-1}^{-1}\cdots h_1^{-1}\setminus J_k}
&={z^\prime_k}|_{S_{n_k,1}\setminus(J_kh_1\cdots h_{k-1})}\\
&={\overline{u}_k}|_{S_{n_k,1}\setminus(J_kh_1\cdots h_{k-1})}\\
&={x_{k,1}}|_{S_{n_k,1}\setminus(J_kh_1\cdots h_{k-1})}\\
&=z|_{S_{n_k,1}\setminus(J_kh_1\cdots h_{k-1})},
\end{align*}
by letting $k\to\infty$ we get
$$x_g=z_{gh_1\cdots h_{l(g)-1}}.$$
Since $g\in G\setminus J$ is arbitrary,
$$x|_{G\setminus J}=(z_{gh_1\cdots h_{l(g)-1}})_{g\in G\setminus J}.$$

\medskip

Now we take $v\in P^J$. Following the same procedure, we find $y\in X$ such that $y|_J=v$ and
$$y|_{G\setminus J}=(z_{gh_1\cdots h_{l(g)-1}})_{g\in G\setminus J}=x|_{G\setminus J}.$$
This completes the proof.
\end{proof}

\medskip

\begin{lemma}\label{increasingdistance}
For every $k\in\mathbb{N}$ there is a continuous mapping
$$f_k:(P^{J_k},d_{l^\infty})\to(X,\rho_{T_k})$$
such that
$$d_{l^\infty}(u,v)\le\rho_{T_k}(f_k(u),f_k(v)),\quad\forall u,v\in P^{J_k}.$$
\end{lemma}
\begin{proof}
We fix $k\in\mathbb{N}$. We take a point $p\in P$. For every $u=(u_g)_{g\in J_k}\in P^{J_k}$ we define $u^\prime\in P^J$ by
$$(u^\prime)_g=\begin{cases}
u_g,&\quad g\in J_k,\\
p,&\quad g\in J\setminus J_k.
\end{cases}$$
Applying Lemma \ref{projection} to $u^\prime\in P^J$, there exists $x(u^\prime)\in X$ such that
$$x(u^\prime)|_J=u^\prime.$$
We define a mapping as follows:
$$f_k:P^{J_k}\to X,\quad u\mapsto x(u^\prime).$$
Notice that for any $u,v\in P^{J_k}$,
$$f_k(u)|_{G\setminus J}=f_k(v)|_{G\setminus J},\quad\quad f_k(u)|_{J_k}=u,\quad f_k(v)|_{J_k}=v,$$$$f_k(u)|_{J\setminus J_k}=u^\prime|_{J\setminus J_k}=v^\prime|_{J\setminus J_k}=f_k(v)|_{J\setminus J_k}.$$
Thus, $f_k$ is continuous. Moreover,
\begin{align*}
\rho_{T_k}\left(f_k(u),f_k(v)\right)
&=\max_{h\in T_k}\rho\left(hf_k(u),hf_k(v)\right)\\
&=\max_{h\in T_k}\sum_{g\in G}\alpha_gd\left(f_k(u)_{gh},f_k(v)_{gh}\right)\\
&\ge\max_{h\in T_k}d\left(f_k(u)_h,f_k(v)_h\right)\quad\quad(\text{since}\;\alpha_e=1)\\
&\ge\max_{h\in J_k}d\left(f_k(u)_h,f_k(v)_h\right)\quad\quad(\text{since}\;J_k\subset T_k)\\
&=\max_{h\in J_k}d\left(u_h,v_h\right)\\
&=d_{l^\infty}\left(u,v\right).
\end{align*}
\end{proof}

\medskip

We are now able to deal with $\mdim(X,\sigma)$ from below. By Lemma \ref{increasingdistance}, we know that for any $\epsilon>0$ and any $k\in\mathbb{N}$,
$$\Widim_\epsilon\left(X,\rho_{T_k}\right)\ge\Widim_\epsilon\left(P^{J_k},d_{l^\infty}\right).$$
By Theorem \ref{widimcube} and the fact that $[0,1]\subset P$\footnote{Strictly speaking, $[0,1]$ is topologically embedded in $P$.}, we have
\begin{align*}
\mdim(X,\sigma)
&=\lim_{\epsilon\to0}\lim_{k\to\infty}\frac{\Widim_\epsilon\left(X,\rho_{T_k}\right)}{|T_k|}\\
&\ge\lim_{\epsilon\to0}\lim_{k\to\infty}\frac{\Widim_\epsilon\left(P^{J_k},d_{l^\infty}\right)}{|T_k|}\\
&\ge\lim_{\epsilon\to0}\lim_{k\to\infty}\frac{\Widim_\epsilon\left([0,1]^{J_k},d_{l^\infty}\right)}{|T_k|}\\
&=\lim_{k\to\infty}\frac{|J_k|}{|T_k|}.
\end{align*}
It follows from (C.k.3) that
$$\frac{1+\delta_k}{2}<\frac{|x_{k,1}(S_{n_k,1},\ast)|}{|S_{n_k,1}|}=\frac{|J_k|}{|T_k|}\le\frac{1+\delta_k}{2}+\frac{1}{|T_k|}.$$
Since $k\in\mathbb{N}$ is arbitrary, we obtain
$$\lim_{k\to\infty}\frac{|J_k|}{|T_k|}=\frac12.$$
Thus, $\mdim(X,\sigma)\ge1/2$.

\medskip

So we finally conclude
$$\mdim(X,\sigma)=\frac12.$$

\bigskip

\noindent\textbf{Part 4: $(X,\sigma)$ cannot be embedded in the full $G$-shift on $[0,1]^G$.}
We shall denote by $([0,1]^G,\sigma^\prime)$ the full $G$-shift on $[0,1]^G$. To complete the whole proof, it remains to show that $(X,\sigma)$ cannot be embedded in $([0,1]^G,\sigma^\prime)$.

\medskip

Recall that $\rho$ and $\rho^\prime$ are the metrics on $P^G$ and $[0,1]^G$, respectively. We assume that there is an embedding
$$f:(X,\sigma)\to([0,1]^G,\sigma^\prime).$$
The paper will end with a contradiction.

\medskip

As $f^{-1}:f(X)\to X$ is a homeomorphism, we fix $\epsilon>0$ such that
$$\rho^\prime(f(x),f(y))<\epsilon\;\text{ implies }\;\rho(x,y)<\frac13,\quad\forall x,y\in X.$$
Since $f\circ\sigma=\sigma^\prime\circ f$, we deduce that
$$\rho_{T_k}^\prime(f(x),f(y))<\epsilon\,\text{ implies }\,\rho_{T_k}(x,y)<\frac13,\quad\forall k\in\mathbb{N},\;\forall x,y\in X.$$
We take $N\in\mathbb{N}$ sufficiently large such that
$$x|_{F_N}=y|_{F_N}\,\text{ implies }\,\rho^\prime(x,y)<\epsilon,\quad\forall x,y\in[0,1]^G.$$
It follows that
$$x|_{F_NT_k}=y|_{F_NT_k}\,\text{ implies }\,\rho_{T_k}^\prime(x,y)<\epsilon,\quad\forall k\in\mathbb{N},\;\forall x,y\in[0,1]^G.$$
For any $k\in\mathbb{N}$ we let
$$\pi_{F_NT_k}:[0,1]^G\to[0,1]^{F_NT_k}$$
be the canonical projection mapping. Consider the mapping
$$\pi_{F_NT_k}\circ f:(X,\rho_{T_k})\to[0,1]^{F_NT_k}.$$
Clearly, $\pi_{F_NT_k}\circ f:(X,\rho_{T_k})\to[0,1]^{F_NT_k}$ is a $(1/3)$-embedding for every $k\in\mathbb{N}$. By Lemma \ref{increasingdistance}, we deduce that
$$\pi_{F_NT_k}\circ f\circ f_k:(P^{J_k},d_{l^\infty})\to[0,1]^{F_NT_k}$$
becomes a $(1/3)$-embedding for every $k\in\mathbb{N}$. It follows from Theorem \ref{epsilonembedding} that
$$|F_NT_k|\ge2|J_k|,\quad\forall k\in\mathbb{N}.$$
However, by (B.k.2) and (C.k.3) we have
$$(1+\delta_k)\cdot|S_{n_k,1}|<2|J_k|\le|F_NT_k|=|F_NS_{n_k,1}|\le|F_{k-1}S_{n_k,1}|<(1+\delta_k)\cdot|S_{n_k,1}|$$
for all $k>N$, a contradiction. Thus, we conclude.

\bigskip

\end{document}